\documentclass[reqno, 12pt]{amsart}
\usepackage{amsmath,amscd,graphics,graphicx,color,a4wide,hyperref,verbatim}
\usepackage{mathptmx}
\usepackage{multicol, fullpage}
\usepackage[usenames,dvipsnames]{xcolor} 
\usepackage{tikz, tikz-3dplot, pgfplots}
\usepackage{tkz-graph}
\usepackage{tikz-cd}
\usetikzlibrary[positioning,patterns] 
\usetikzlibrary{matrix,arrows,decorations.pathmorphing}

\usepackage{textcmds}

\usepackage{graphicx}
\usepackage{psfrag}
\usepackage{marvosym}
\usepackage{amsmath}
\usepackage{amsfonts}
\usepackage{amssymb}
\usepackage{amsthm}
\usepackage{mathrsfs}
\usepackage{enumitem}
\usepackage{subcaption}
\usepackage{ytableau}

\usepackage{calligra}
\usepackage{mathrsfs}
\DeclareMathAlphabet{\mathcalligra}{T1}{calligra}{m}{n}
\DeclareFontShape{T1}{calligra}{m}{n}{<->s*[1.5]callig15}{}

\newtheorem{theorem}{Theorem}[section]
\newtheorem{lemma}[theorem]{Lemma}
\newtheorem{lem-def}[theorem]{Lemma-definition}
\newtheorem{proposition}[theorem]{Proposition}
\newtheorem{prop-def}[theorem]{Proposition-definition}
\newtheorem{claim}[theorem]{Claim}

\newtheorem*{conjecture}{DK Conjecture}

\theoremstyle{definition}

\newtheorem{remark}[theorem]{Remark}
\numberwithin{equation}{section}

\newtheorem{thm}{Theorem}[section] 
\theoremstyle{plain} 
\newcommand{\thistheoremname}{}
\newtheorem{genericthm}[thm]{\thistheoremname}
  
\newtheorem*{genericthm*}{\thistheoremname}
\newenvironment{namedthm*}[1]
  {\renewcommand{\thistheoremname}{#1}%
   \begin{genericthm*}}
  {\end{genericthm*}}

\newcommand{\CC} {\mathbb{C}}

\newcommand{\LL} {\mathbb{L}}

\newcommand{\PP} {\mathbb{P}}

\newcommand{\RR} {\mathbb{R}}

\newcommand{\ZZ} {\mathbb{Z}}

\newcommand {\shA} {\mathcal{A}}
\newcommand {\shB} {\mathcal{B}}
\newcommand {\shC} {\mathcal{C}}
\newcommand {\shD} {\mathcal{D}}
\newcommand {\shE} {\mathcal{E}}
\newcommand {\shF} {\mathcal{F}}

\newcommand {\shH} {\mathcal{H}}

\newcommand {\shO} {\mathcal{O}}

\newcommand {\shS} {\mathcal{S}}
\newcommand {\shT} {\mathcal{T}}

\newcommand {\Ext} {\operatorname{Ext}}
\newcommand{\sExt}{\mathscr{E} \kern -3pt xt}

\newcommand {\Hom} {\operatorname{Hom}}
\newcommand {\sHom}{\mathscr{H}\kern-5pt\mathcalligra{om}}

\newcommand {\Id} {\operatorname{Id}}

\usepackage{xcolor,colortbl}

\usepackage{array}
\newcolumntype{L}[1]{>{\raggedright\let\newline\\\arraybackslash\hspace{0pt}}m{#1}}
\newcolumntype{C}[1]{>{\centering\let\newline\\\arraybackslash\hspace{0pt}}m{#1}}
\newcolumntype{R}[1]{>{\raggedleft\let\newline\\\arraybackslash\hspace{0pt}}m{#1}}

\DeclareRobustCommand{\Sec}{\ifmmode\mathsection\else\textsection\fi}

\makeatletter
\newcommand\xleftrightarrow[2][]{%
  \ext@arrow 9999{\longleftrightarrowfill@}{#1}{#2}}
\newcommand\longleftrightarrowfill@{%
  \arrowfill@\leftarrow\relbar\rightarrow}
\makeatother

\title{DK Conjecture for Some $K$-inequivalences from Grassmannians}
\author{Naichung Conan Leung and Ying Xie}

\address{The Institute of Mathematical Sciences and Department of Mathematics,
The Chinese University of Hong Kong, Shatin, N.T., Hong Kong}\email{leung@math.cuhk.edu.hk}

\address{Department of Mathematics, Science School, Southern University of Science and Technology,
Shenzhen, China}\email{xiey@sustech.edu.cn}

\begin{document}

\begin{abstract}
The DK conjecture of Bondal-Orlov \cite{bondal2002derived} and Kawamata \cite{kawamata2002d} states that there should be an embedding of bounded derived categories for any $K$-inequivalence, which is proved to be true for the toroidal case (\cite{kawamata20054}, \cite{kawamata2006}, \cite{kawamata2013} and \cite{Kawamata2016}). In this paper, we construct examples of non-toroidal $K$-inequivalences from Grassmannians inspired by \cite{kuznetsov2018derived}, \cite{ueda2019g_2}, \cite{morimura2021derived} and \cite{kanemitsu2022mukai}, and we show that these $K$-inequivalences satisfy the DK conjecture. 
\end{abstract}

\maketitle

\section{Introduction}
The bounded derived categories of coherent sheaves $D^b Coh(X)$ (or simply $D(X)$ throughout this paper) on a smooth projective variety $X$ encodes a lot of geometric information. For instance, the dimension and the canonical ring of $X$ can be determined from $D(X)$. The canonical divisor $K_X$ is one of the most important invariants of $X$. The change of $K_X$ will appear if we run the minimal model program. Also, the Serre functor of $D(X)$ is isomorphic to $-\otimes \shO(K_X)[\dim X]$.  
 
Recall that a birational map $f: X_2\dashrightarrow X_1$ between two smooth projective varieties $X_1$ and $X_2$ is called a \emph{$K$-inequivalence} if there is a third smooth projective variety $X$ with two birational morphisms $\pi_1: X\rightarrow X_1$ and $\pi_2: X\rightarrow X_2$ such that $f=\pi_1\circ \pi_2^{-1}$ and $\pi_2^*K_{X_2}=\pi_1^*K_{X_1}+D$ for some effective divisor $D$ on $X$.  If $\pi_2^*K_{X_2}=\pi_1^*K_{X_1}$, then $f$ is called a \emph{$K$-equivalence}. 

\begin{equation*}
 \begin{tikzcd}
  &\arrow[dl,swap, "\pi_2"]  X \arrow[dr, "\pi_1"]\\
   X_2 \arrow[rr,dashed, "f"]     && X_1.
\end{tikzcd}
\end{equation*} 

\begin{conjecture}[Bondal-Orlov \cite{bondal2002derived} and Kawamata \cite{kawamata2002d}]\label{kbo}
 
For any $K$-inequivalence, 
\begin{equation*}
\begin{tikzcd}
 X_2 \arrow[r,dashed]   & X_1,
\end{tikzcd}
\end{equation*}
 there is an exact fully-faithful embedding of triangulated categories: 
\begin{equation*}
\begin{tikzcd}
 D(X_1) \arrow[r,hook]   & D(X_2).
\end{tikzcd}
\end{equation*}
For any $K$-inequivalence, $X_2\dashrightarrow X_1$, there is an equivalence $D(X_2)\cong D(X_1)$. 
\end{conjecture}

Unlike the situations for $K$-equivalences (see \cite{kawamata2017birational} for the survey of DK Conjecture), there are few examples of $K$-inequivalences proven to satisfy the DK Conjecture except for some toroidal type (see \cite{kawamata20054}, \cite{kawamata2006}, \cite{kawamata2013} and \cite{Kawamata2016}). In this paper, we construct examples of non-toroidal type $K$-equivalences whose exceptional locus are Grassmannians and show that these $K$-inequivalences satisfy the DK Conjecture. We now describe the examples of $K$-inequivalences. 

Start with the partial flag variety $$Fl(1,2,N)=\{(V_1, V_2)|\, V_1\subset V_2\subset \CC^N, \,\text{dim } V_1=1, \text{dim } V_2=2\}.$$ It admits two projective space fibrations onto $Gr(1, N)=\PP^{N-1}$ and $Gr(2, N)$, respectively:
\begin{center}
\begin{tikzcd}
                       & \arrow[dl, swap, "p_2"]  Fl(1,2,N)  \arrow[dr,"p_1"] \\
                         Gr(2, N) &  & \PP^{N-1}.
\end{tikzcd}
\end{center}
Note that $Fl(1,2,N)\cong \PP_{Gr(2,N)}(U)\cong \PP_{\PP^{N-1}}(Q)$, where $U$ is the tautological rank 2 subbundle on $Gr(2,N)$ and $Q$ is the tautological rank $(N-1)$ quotient bundle on  $\PP^{N-1}$. 
Denote the ample generator of $Pic(\PP^{N-1})$ (resp. $Pic(Gr(2,N))$) by $h$ (resp. H).  


It is easy to check that 
\begin{align}\label{projformula}
& p_{2*}\shO(h+H)=U^{\vee}\otimes \shO(H),\\
& p_{1*}\shO(h+H)=Q^{\vee}\otimes\shO(2h).
\end{align}
Let $X_2=\PP_{Gr(2, N)}(U(-H)\oplus\shO)$ and $X_1=\PP_{\PP^{N-1}}(Q(-2h)\oplus \shO)$. Then $X_1$ contains $Gr(1, N)=\PP^{N-1}$ as a closed subvariety by the following natural morphisms
$$\PP^{N-1}\cong \PP_{\PP^{N-1}}(0\oplus \shO)\hookrightarrow \PP_{\PP^{N-1}}(Q(-2h)\oplus \shO).$$ Similarly, $X_2$ contains $Gr(2, N)$ as a closed subvariety. Consider two blowing-ups $Bl_{Gr(2,N)} X_2$ and $Bl_{\PP^{N-1}}X_1$, both of which are isomorphic to $X:=\PP_{Fl(1, 2, N)}(\shO(-h-H)\oplus \shO)$, with the same exceptional divisor $E\cong Fl(1,2,V)$. Therefore, we get a birational map $f$ from $X_2$ to $X_1$, and the geometry can be summarized in the following diagram: 
 \begin{equation}\label{aflip}
 \begin{tikzcd}
 &E=Fl(1,2,N) \arrow[dl,swap,"p_1"] \arrow[d, hook, "j"]\arrow[dr,"p_2"]\\
  Gr(2,N)\arrow[d,hook] & \arrow[dl,swap, "\pi_2"]  X \arrow[dr, "\pi_1"]&\arrow[d, hook',swap] \PP^{N-1}\\
   X_2 \arrow[rr,dashed, "f"]   &  & X_1. 
\end{tikzcd}
\end{equation}

Now, an easy computation implies that
$$D=\pi_2^*K_{X_2}-\pi_1^{*}K_{X_1}=(N-3)E.$$ 

The birational map $f$ is a $K$-inequivalence for $N>3$ and a $K$-equivalence for $N=3$.

\begin{theorem}\label{thm}
The  $K$-inequivalence birational map $f: X_2\dashrightarrow X_1$ in (\ref{aflip}) satisfies the DK Conjecture, i.e., there is a fully-faithful embedding of triangulated categories:
\begin{center}
\begin{tikzcd}
D(X_1) \arrow[r, hook, "\Phi"] & D(X_2). 
\end{tikzcd}
\end{center}
\end{theorem}

\begin{remark}
\
\begin{enumerate}
\item
When $N=3$, the birational map $f: \PP(\Omega_{\PP^{2}}\oplus \shO) \dashrightarrow \PP(\Omega_{\check{\PP}^2}\oplus \shO)$ is the Mukai flop, and the functor $\Phi$ is an equivalence in this case. (see Remark \ref{proofn1} or \cite{morimura2021derived}). 
\item
Theorem \ref{thm} holds for any $K$-inequivalence which locally looks like (\ref{aflip}). 

\item The functor $\Phi$ in Theorem \ref{thm} is a kernel functor induced by a complicated complex of coherent sheaves (even for $N=4$) in $D(X_1\times X_2)$ (see \ref{functor}). It is interesting to give a geometric interpretation of the complex. 

\item The above construction of the $K$-inequivalence can be generalized to any generalized partial flag variety with two different projective space fibrations.  In \cite{cflip}, we construct a series of $K$-inequivalences from generalized partial flag varieties of complex semi-simple algebraic groups. Moreover, we show that these $K$-inequivalences are flips in the minimal model program. In particular, the $K$-inequivalence $f$ above is a flip. 

\item The DK conjecture for $K$-equivalences from ($\PP^n$, $\PP^{n*}$), ($Gr(2, 5), Gr(3, 5)$), $C_2$-Grassmannian pairs and $G_2$ Grassmannians pairs are proved by using similar methods in \cite{kuznetsov2018derived}, \cite{ueda2019g_2}, \cite{morimura2021derived} and \cite{rampazzo2021equivalences}. 

\end{enumerate}
\end{remark}

\subsection*{Conventions}
In this paper, $\PP(V)=Proj (Sym^{\bullet}V^{\vee})$ for any vector bundle $V$. The derived functors $\RR Hom(-,-)$ and $\Ext^{\bullet}(-,-)$ are taken over the total space $X$. We will omit the natural functors $p_1^*, p_2^*$ and $j_*$, $i_{1*}$, $i_{2*}$ if no confusion occurs. 

\subsection*{Strategy of Proof}
Firstly, we can embed $D(X_1)$ and $D(X_2)$ into $D(X)$ by Orlov's blow-up formula \cite{orlov1992projective} so that we have the following two 
semiorthogonal decompositions (SOD) of $D(X)$:
\begin{align}
& D(X)=\langle \pi_1^{*}D(X_2), j_*p_2^* D(Gr(2,N)) \rangle, \label{SOD1}\\
& D(X)=\Big\langle \pi_2^*D(X_1), \langle j_*(p_1^*D(\PP^{N-1}))(kH)\rangle_{0\leq k \leq N-3}\rangle\Big\rangle. \label{SOD2}
\end{align}
Both the left orthogonal complements of $D(X_1)$ and $D(X_2)$ are (copies) of derived categories of Grassmanianns. It is known that both $D(\PP^{N-1})$ and $D(Gr(2,N))$ admit full exceptional collections by \cite{beilinson1978coherent} and \cite{kuznetsov2008exceptional}. Moreover, $D(\PP^{N-1})$ consists of line bundles only while $D(Gr(2,N))$ involves $S^k U$, the symmetric powers of the tautological subbundle $U$ on $Gr(2, N)$.

Secondly, we use the SOD of $D(Gr(2,N))$ to simplify (\ref{SOD1}) to the form (\ref{SOD1mut}) :
\begin{align*}
D(X)=\Big \langle \shD, \langle \shO(\ell h)\rangle_{-1\leq \ell\leq n-1},\shH, \langle\shF_{\ell}\rangle_{0\leq \ell\leq n-2}, S^{n-1} U(H), \shA(\ell H)\rangle_{2\leq \ell\leq 2n-2}\Big \rangle 
\end{align*}
via mutation techniques by Kuznetsov in \cite{kuznetsov2010derived} (C.f. also \cite{kuznetsov2018derived}, \cite{ueda2019g_2} or \cite{morimura2021derived}). Here $\shD$ is equivalent to $D(X_2)$. However, there are still lots of symmetric powers of $U$ remaining in the left orthogonal complement of $\shD$, which does not happen for the flop situation (see remark \ref{proofn1}).    

Finally, to get rid of the remaining $S^kU$'s, we apply the chess-game method introduced in \cite{thomas2018notes} (also in \cite{jiang2021categorical}) to prove Theorem \ref{thm}.  Roughly speaking, the chess-game method is an analogy of the spectral sequence argument in cohomology theory. It is a systematic method to do cancellation of categories.

The rest of this paper is organized as follows. Section 2 lists all cohomology and mutations that will be used in the proof of Theorem \ref{thm}. 
Section 3 presents the mutation process on $D(X)$ for the case when $N$ is odd. Section 4 completes the detailed proof of Theorem  \ref{thm} for the case when $N$ is odd. Section 5 sketches the proof of Theorem \ref{thm} for the case when $N$ is even. The Appendix contains background knowledge on mutations and the Borel-Weil-Bott Theorem that is used in this paper. 

\subsection*{Acknowledgements}
We thank Kowk Wai Chan, Jesse Huang, Lisa Li, Laurent Manivel, Yukinobu Toda,  Zhiwei Zheng, Yan Zhou, and especially Yalong Cao, Qingyuan Jiang, Mikhail Kapranov, Yujiro Kawamata, Eduard Looijenga, Chin-Lung Wang for many helpful discussions and suggestions when preparing this paper. The authors are supported by grants from the Research Grants Council of the Hong Kong Special Administrative Region, China (Project No. CUHK14301117 and CUHK14303518). 

\section{Vanishing of Cohomology and Mutations}
In this section, we list all vanishing results and mutations that will be used later in the subsequent sections. For $\ell\geq 0$, let 
\begin{align*}
& \shA^{l}=\langle \shO, U^{\vee}, \cdots, S^{n-l-1}U^{\vee} \rangle;; \\
&\shA_{l}=\langle S^{l} U^{\vee}, S^{l+1}U^{\vee}, \cdots, S^{n-1}U^{\vee}\rangle\\
& \shB_{\ell}=\langle \shO((\ell+1)h), S^{\ell}U^{\vee}(H-h)\rangle;\\
&\shC_{\ell}=\langle \shO(\ell h),S^{\ell}U^{\vee}(H-2h)\rangle;\\  
& \shE_{\ell}=\langle \shO(\ell h),S^{\ell-1}U^{\vee}(H-h),\shO((\ell+1)(H-h)-h)\rangle;\\
&\shF_{\ell}=\begin{cases} 
\langle S^{\ell}U^{\vee}(H), \shO((\ell+2)(H-h)),\shO((\ell+3)(H-h)-h)\rangle &\mbox{if } \ell\leq n-4;\\
\langle S^{\ell}U^{\vee}(H), \shO((\ell+2)(H-h))\rangle &\mbox{if } \ell>n-4;\\
\end{cases}\\
&\shH=\begin{cases} 
\langle \shO(H-2h), \shO(H-h)\rangle & \mbox{if } n=2;\\
\langle \shO(H-2h), \shO(H-h), \shO(2H-3h)\rangle &\mbox{if } n\geq 3;
 \end{cases}
\end{align*}



\begin{lemma}[Kapranov \cite{kapranov1988derived} and Kuznetsov \cite{kuznetsov2008exceptional}]  $D(Gr(2,N))$ admits a full exceptional collections: 
\begin{align}\label{SODGr}
D(Gr(2, N))=\begin{cases} \Big\langle\langle \shA^1(kH)\rangle_{0\leq k\leq n-1}, \langle \shA(\ell H)\rangle_{n-2\leq \ell\leq 2n-1}\Big \rangle &\mbox{if } N=2n,\\
 \langle \shA(kH)\rangle_{0\leq k\leq N-1}&\mbox{if } N=2n+1,\end{cases}
 \end{align}
 \end{lemma}

\begin{lemma}\label{van}
On the total space $X$, we have the following vanishing results. 
\begin{enumerate}
\item \label{1ind}
For any $0\leq k\leq n-1$, 
 $$\RR Hom(S^{n-k-1}U^{\vee}(H-h), \shA^k)=0.$$
\item \label{2ind} For $0\leq k \leq n-1$, 
\begin{align*}
&\RR Hom(\shA_{k+2}(-h), \shO(kh))=0;\\
& \RR Hom(\shA^{n-k-1}(H-h), \shO(kh))=0.
\end{align*} 
\item \label{3ind} For $1\leq k< \ell\leq n-2$,   
$$\RR Hom(\shC_{\ell}, S^{k-1}U^{\vee}(H-h))=0.$$         
\item \label{4ind} For $1\leq k\leq n-2$ and $n-k\leq \ell\leq n-1$,
\begin{align*}
&\RR Hom(S^{n-2-k} U^{\vee}(H-h), \shO(\ell H))=0;\\
&\RR Hom(S^{n-k} U^{\vee}(H-h), \shA^{k+1}(H))=0;\\
&\RR Hom(\shO((n-k)(H-h)-h), \shA^{k+2}(H))=0.\\
\end{align*}
\vspace{-1cm}
\item\label{step3} For any $0\leq \ell<k \leq r$, $r=\left[(n-1)/2\right]$, i.e., the greatest integer no bigger than $\left[(n-1)/2\right]$,  
$$
\RR Hom(\shA_{n-2\ell-1}((n+\ell)H), \shA^{2k+1}((n+k)H))=0.
$$
\item\label{ab} Either (i) $2\leq a-b \leq 2n-3$ and $b>0$ or (ii) $b<0$, 
$$\Ext^{\bullet}(\shO(ah), \shO(bH))=0.$$ 
\end{enumerate}
\end{lemma}

\begin{proof}We give proof of (1) here, and the rest can be obtained using the same arguments (See Appendix \ref{proofofvan} for detailed proof of others).  By adjunction of pullback-pushforward,
\[
\RR Hom(S^{n-k-1}U^{\vee}(H-h), \shA^k)=\RR Hom_E(\LL j^*j_* S^{n-k-1}U^{\vee}(H-h), \shA^k).
\]
Recall there is a distinguished triangle associated to the closed immersion $j: \,E\hookrightarrow X$:
\begin{align}\label{att}
  \shO(H+h)[1]= \shO_E(-E)[1]\longrightarrow \LL j^*j_* \longrightarrow id \xrightarrow{[1]}
\end{align} 
and it induces a distinguished triangle of complex of vector spaces:
\begin{align*}
& \RR Hom_E(S^{n-k-1}U^{\vee}(H-h), \shA^k) \longrightarrow \RR Hom( S^{n-k-1}U^{\vee}(H-h), \shA^k)\\
& \longrightarrow  \RR Hom_E(S^{n-k-1}U^{\vee}(2H)[1], \shA^k) \xrightarrow{[1]}.
\end{align*}
So, it is sufficient to show the vanishing of the first and third terms. With the help of the projection formula (\ref{projformula}) or Lemma \ref{pfp}, we have the following ($0\leq a \leq n-k-1$):
\begin{align*}
(i) &\,\Ext^{\bullet}_E(S^{n-k-1}U^{\vee}(2H), S^aU^{\vee} )=\Ext^{\bullet}_{Gr(2, N)}(S^{n-k-1}U^{\vee}(2H), S^aU^{\vee}),\\
(ii) &\,\Ext^{\bullet}_E(S^{n-k-1}U^{\vee}(H-h), S^aU^{\vee})=\Ext^{\bullet}_{Gr(2,N)}(S^{n-k-1}U^{\vee}(H), S^aU^{\vee}\otimes U^{\vee})\\
&=\Ext^{\bullet}_{Gr(2,N)}(S^{n-k-1}U^{\vee}(H), S^{a+1}U^{\vee}\oplus S^{a-1}U^{\vee}(H) ). \end{align*}
It is noted that all these vanish by the SOD of $D(Gr(2, N))$ (\ref{SODGr}) except for the case $k=0, a=n-1$:
$$ \Ext^{\bullet}_{Gr(2,N)}(S^{n-1}U^{\vee}(H), S^{n}U^{\vee})=0.$$
By Littlewood-Richardson rule,  $$S^{n-1}U^{\vee}(-H)\otimes S^{n}U^{\vee} =\bigoplus_{t=0}^{n-1}\Sigma^{n-t-1, -n+t}U^{\vee}.$$
Then $$\Ext^{\bullet}_{Gr(2,N)}(S^{n-1}U^{\vee}(H), S^{n}U^{\vee})=\bigoplus_{t=0}^{n-1}H^{\bullet}(Gr(2,N), \Sigma^{n-t-1, -n+t}U^{\vee})=0$$
by  checking the criterion of Theorem \ref{BWB} (BWB).
\end{proof}

For the projection $p_2$, there exists a relative Euler sequence on $E=\PP_{Gr(2,N)}(U)$:
\begin{align*}
0 \longrightarrow \shO(H-h) \longrightarrow U^{\vee} \longrightarrow \shO(h) \longrightarrow 0.
\end{align*}
For each $k$, there are two short exact sequences on $E$ and $X$ by taking $k$-th symmetric power :
\begin{align}\label{euler}
& 0 \longrightarrow S^{k-1}U^{\vee}(H-h) \longrightarrow S^kU^{\vee} \longrightarrow  \shO(kh)
 \longrightarrow 0;\\
& 0  \longrightarrow \shO(k(H-h)) \longrightarrow S^kU^{\vee} \longrightarrow  S^{k-1}U^{\vee}(kh)\longrightarrow 0.
\end{align}
Not surprisingly, these two induce the following mutations (See Appendix \ref{proofofmut} for the proof):
\begin{lemma}\label{mut}For any $1\leq k\leq n-1$, 
\begin{align*}
&(1)\,\, \LL_{S^{k-1}U^{\vee}(H-h)}S^kU^{\vee}=\shO(kh);\\ 
&(2)\, \,\RR_{\shO(kh)} S^{k}U^{\vee} = S^{k-1}U^{\vee}(H-h);\\
&(3) \,\, \RR_{S^{k-1}U^{\vee}(-h)} S^kU^{\vee}=\shO(k(H-h)).
\end{align*}
\end{lemma}

\section{Mutations on Derived category of $X\, (N=2n+1)$}
In this section, we simplify the SOD of $D(X)$ (\ref{SOD1}) by mutation techniques when $N=2n+1$ ($N=2n$ case will be explained in section \ref{even}).\footnote{The readers can refer to Appendix \ref{appa} for the preliminaries and background knowledge of left and right mutations.}   
The main result of this section is 
\begin{proposition} For $n\geq 2$, 
\begin{align}\label{SOD1mut}
D(X)=\Big \langle \shD, \langle \shO(\ell h)\rangle_{-1\leq \ell\leq n-1},\shH, \langle\shF_{\ell}\rangle_{0\leq \ell\leq n-2}, S^{n-1} U(H), \shA(\ell H)\rangle_{2\leq \ell\leq 2n-2}\Big \rangle 
\end{align}
where $\shD=\LL_{\langle \shA(-h), \shA(H-h)\rangle} \pi_2^*D(X_2).$
\end{proposition}

\begin{proof}
At first we left mutate $\langle \shA((2n-1)H), \shA(2nH)\rangle$ to the far left. Note that $$\LL_{\langle\shA((2n-1)H), \shA(2nH)\rangle^{\perp}}\bigg|_{\langle\shA((2n-1)H), \shA(2nH)\rangle}=-\otimes K_{X}[dim\,X]$$ and $K_{X}\big|_E=\shO(-h-(2n-1)H)$ (see Lemma \ref{lem:mut}). Thus
\begin{align*}
D(X)=\Big \langle \shA(-h), \shA(H-h), \pi_2^*D(X_2), \langle \shA(kH)\rangle_{0\leq k\leq 2n-2}\Big\rangle.
\end{align*}
Then left mutate $\pi_2^*D(X_2)$ through $\langle \shA(-h), \shA(H-h)\rangle$:
\begin{align}\label{SOD1.1}
D(X)=\Big \langle \shD, \shA(-h), \shA(H-h),\langle \shA(kH)\rangle_{0\leq k\leq 2n-2} \Big\rangle.
\end{align}

\indent To reach (\ref{SOD1mut}), we need to apply mutation through $\langle \shA(-h), \shA(H-h), \shA, \shA(H)\rangle$ which is involved and consists of the four inductive steps. 
\begin{enumerate}
\item \textbf{First Inductive Step ($n\geq 1$)}: Mutation on $\langle \shA(H-h), \shA\rangle$.
\begin{lemma}\label{1ms}
For $1\leq k \leq n$,
\begin{align}\label{sod3.3}
\langle \shA(H-h), \shA\rangle= \Big\langle \shA^{k}(H-h), \shA^{k-1}, \langle \shB_{\ell}\rangle_{n-k \leq \ell\leq n-2}, S^{n-1}U^{\vee}(H-h)\Big\rangle.
\end{align}
\end{lemma}
\begin{proof}
Prove by induction on $k$. 
\begin{itemize}
\item Base case ($k=1$). We can exchange $S^{n-1}U^{\vee}(H-h)$ and  $\shA$ by Lemma \ref{1ind} ($k=0$): $$LHS=\langle \shA^1(H-h), S^{n-1}U^{\vee}(H-h), \shA\rangle=\langle \shA^1(H-h), \shA, S^{n-1}U^{\vee}(H-h)\rangle=RHS.$$ 
\item Assume that we have the SOD (\ref{sod3.3}) for case $k$. Then

\begin{align*}
RHS& = \Big\langle \shA^{k+1}(H-h), \underline{S^{n-k-1}U^{\vee}(H-h), \shA^{k}}, S^{n-k}U^{\vee}, \langle \shB_{\ell}\rangle_{n-k \leq \ell\leq n-2}\Big\rangle\\
&=\Big\langle  \shA^{k+1}(H-h), \shA^{k}, \underline{S^{n-k-1}U^{\vee}(H-h), S^{n-k}U^{\vee}}, \langle \shB_{\ell}\rangle_{n-k \leq \ell\leq n-2}\Big\rangle\\
&=\Big\langle \shA^{k+1}(H-h), \shA^{k}, \langle \shB_{\ell}\rangle_{n-k-1 \leq \ell\leq n-2}\Big\rangle.
\end{align*}
This is just the case $k+1$, and the lemma follows.  In the second line, we exchange $S^{n-k-1}U^{\vee}(H-h)$ and  $\shA^{k}$ by Lemma \ref{van} (\ref{1ind}), and in the last line 
we left mutate $S^{n-k}U^{\vee}$ through $S^{n-k-1}U^{\vee}(H-h)$ (Lemma \ref{mut}). 
\end{itemize}
\end{proof}
Apply Lemma \ref{1ms} for the final case $k=n$:
\begin{align}\label{1mut}
\langle \shA(H-h), \shA\rangle=\Big\langle \shO, \langle \shB_{\ell}\rangle_{0\leq \ell \leq n-2}, S^{n-1}U^{\vee}(H-h)\Big\rangle. 
\end{align}
\begin{remark}
When $n=1$, (\ref{1mut}) is nothing but $$\langle \shO(H-h), \shO\rangle=\langle \shO, \shO(H-h)\rangle.$$
\end{remark}

\item \textbf{Second Inductive Step ($n\geq 2$)}: Mutation on $\Big\langle \shA_1(-h), \shO, \langle \shB_{\ell}\rangle_{0\leq \ell \leq n-2}\Big\rangle$. 
\begin{lemma}\label{2ms}
For any $1\leq k\leq n-1$,
\begin{align}\label{sod3.5}
\Big\langle \shA_1(-h), \shO, \langle \shB_{\ell}\rangle_{0\leq \ell \leq n-2}\Big\rangle=\Big\langle  \langle \shC_{\ell}\rangle_{0\leq \ell \leq k-1}, \shA_{k+1}(-h), \shA^{n-k+1}(H-h),\langle \shB_{\ell}\rangle_{k-1\leq \ell\leq n-2}\Big\rangle.  
\end{align}
\end{lemma}

\begin{proof}
Prove by induction on $k$:
\begin{itemize}
\item Base case $k=1$ is equivalent to $\langle U^{\vee}(-h), \shA_2(-h), \shO\rangle=\langle  \shO, \shO(H-2h),\shA_2(-h)\rangle$, which can be proved by firstly exchange 
$\shO$ and $\shA_2(-h)$ using Lemma \ref{van} (\ref{2ind}) ($k=0$) and right mutate $U^{\vee}(-h)$ through $\shO$ (Lemma \ref{mut}). 
\item Assume we have the SOD (\ref{sod3.5}) for case $k$. Then
\begin{align*}
RHS&=\Big\langle \langle \shC_{\ell}\rangle_{0\leq \ell \leq k-1}, S^{k+1}U^{\vee}(-h), \underline{\shA_{k+2}(-h), \shA^{n-k-1}(H-h), \shO(kh)}, S^{k-1}U^{\vee}(H-h), \langle \shB_{\ell}\rangle_{k\leq \ell\leq n-2}\Big\rangle\\ 
&=\Big\langle \langle \shC_{\ell}\rangle_{0\leq \ell \leq k-1}, \underline{S^{k+1}U^{\vee}(-h), \shO(kh)}, \shA_{k+2}(-h), \shA^{n-k}(H-h), \langle \shB_{l}\rangle_{k\leq \ell\leq n-2}\Big\rangle\\ 
&=\Big\langle \langle \shC_{\ell}\rangle_{0\leq \ell \leq k}, \shA_{k+2}(-h), \shA^{n-k}(H-h), \langle \shB_{\ell}\rangle_{k\leq \ell\leq n-2}\Big\rangle.
\end{align*}
This is the case $k+1$. In the second line we exchange $\shA_{k+2}(-h)$ and $\langle \shA^{n-k-1}(H-h), \shO(kh)\rangle $ by Lemma \ref{van} (\ref{2ind}), and in the third line we right mutate $S^{k+1}U^{\vee}(-h)$ through $\shO(kh)$ (Lemma \ref{mut}). 
\end{itemize}
\end{proof}
Apply Lemma \ref{2ms} to the final case ($k=n-1$):
\begin{align*}
\Big\langle \shA_1(-h), \shO, \langle \shB_{\ell}\rangle_{0\leq \ell \leq n-2}\Big\rangle=\Big\langle  \langle \shC_{\ell}\rangle_{0\leq \ell \leq n-2}, \shA^{2}(H-h),\shB_{n-2}\Big\rangle.  
\end{align*}
\item \textbf{Third Inductive Step ($n\geq 3$):} Mutation on $\Big\langle\langle \shC_{\ell}\rangle_{1\leq \ell \leq n-2}, \shA^{2}(H-h)\Big\rangle.$ 
\begin{lemma}\label{3ms}
For any $1\leq k\leq n-1$,
\begin{align}\label{sod3.6}
\Big\langle\langle \shC_{\ell}\rangle_{1\leq \ell \leq n-2}, \shA^{2}(H-h)\Big\rangle=\Big\langle\langle \shE_{\ell}\rangle_{1\leq \ell\leq k-1}, \langle \shC_{\ell}\rangle_{k\leq \ell \leq n-2}, \shA_{k-1}^2(H-h)\Big\rangle. 
\end{align}
\end{lemma}
\begin{proof}
Prove by induction on $k$
\begin{itemize}
\item  Base case $k=1$ is trivial. 
\item Assume that we have the SOD (\ref{sod3.6}) for the case $k$. Then 
\begin{align*}
RHS&=\Big\langle \langle \shE_{\ell}\rangle_{1\leq \ell\leq k-1}, \shO(kh), S^{k}U^{\vee}(H-2h), \underline{\langle \shC_{\ell}\rangle_{k+1\leq \ell \leq n-2}, S^{k-1}U^{\vee}(H-h)}, \shA_{k}^2(H-h)\Big\rangle\\
&=\Big\langle \langle \shE_{\ell}\rangle_{1\leq \ell\leq k-1}, \shO(kh), \underline{S^{k}U^{\vee}(H-2h), S^{k-1}U^{\vee}(H-h)}, \langle \shC_{\ell}\rangle_{k+1\leq \ell \leq n-2}, \shA_{k}^2(H-h)\Big\rangle\\
&=\Big\langle \langle \shE_{\ell}\rangle_{1\leq \ell\leq k}, \langle \shC_{\ell}\rangle_{k+1\leq \ell \leq n-2}, \shA_{k}^2(H-h)\Big\rangle.
\end{align*}
This is the case $k+1$. In the second line we exchange $\langle \shC_{\ell}\rangle_{k+1\leq \ell \leq n-2}$ and $S^{k-1}U^{\vee}(H-h)$ by Lemma \ref{van} (\ref{3ind}),  and in the third line we right mutate $S^{k}U^{\vee}(H-2h)$ through $S^{k-1}U^{\vee}(H-h)$ (Lemma \ref{mut}).
\end{itemize}
\end{proof}
Apply Lemma \ref{3ms} to the final case ($k=n-1$):
\[
\Big \langle\langle \shC_{\ell}\rangle_{1\leq \ell \leq n-2}, \shA^{2}(H-h)\Big\rangle=\langle \shE_{\ell}\rangle_{1\leq \ell\leq n-2}. 
\]
Lastly, in the third inductive process, we do mutation on $\langle S^{n-1}U^{\vee}(H-h), \shA^1(H)\rangle$ by first exchanging $S^{n-1}U^{\vee}(H-h)$ and $\shA^2(H)$ by Lemma \ref{van} (\ref{4ind}) and then right mutating $S^{n-1}U^{\vee}(H-h)$ through $S^{n-2}U^{\vee}(H)$:
 \begin{align*}
 \langle S^{n-1}U^{\vee}(H-h), \shA^1(H)\rangle=\langle \shA^2(H), \shF_{n-2}\rangle.
 \end{align*}
 
 \item \textbf{Forth Inductive Process ($n\geq 3$)}:Mutation on $\Big\langle \langle \shE_{\ell}\rangle_{1\leq \ell\leq n-2}, \shB_{n-2},\shA^2(H)\Big\rangle$. 
\begin{lemma}\label{4ms}
For any $1\leq k\leq n-2$, 
\begin{align}\label{sod3.7}
&\Big\langle \langle \shE_{\ell}\rangle_{1\leq \ell\leq n-2}, \shB_{n-2}, \shA^2(H)\Big\rangle=\Big\langle\langle \shE_{\ell}\rangle_{1\leq \ell\leq n-1-k}, \langle\shO(\ell H)\rangle_{n-k\leq \ell \leq n-1}, \shA^{k+2}(H), \langle\shF_{\ell}\rangle_{n-2-k\leq \ell\leq n-3}\Big\rangle. 
\end{align}
\end{lemma}
\begin{proof}
Prove by induction on $k$. 
\begin{itemize}
\item Base case $k=1$ is equivalent to
\[
\langle \shB_{n-2}, \shA^2(H)\rangle=\langle  \shO((n-1)h), \shA^3(H), \shF_{n-3}\rangle.
\]
We can exchange $S^{n-2}U^{\vee}(H-h)$ and $\shA^3(H)$ by Lemma \ref{van}(\ref{4ind}) and then right mutate $S^{n-2}U^{\vee}(H-h)$ through $S^{n-3}U^{\vee}(H)$ (Lemma \ref{mut}). That is,  
\begin{align*}
 LHS&=\langle \shO((n-1)h), \shA^3(H), S^{n-2}U^{\vee}(H-h), S^{n-3}U^{\vee}(H)\rangle=RHS.
\end{align*}
 
\item Assume that we have the SOD (\ref{sod3.7}) for the case $k$. Then
\begin{align*}
RHS
&=\Big\langle\langle \shE_{\ell}\rangle_{1\leq \ell\leq n-2-k}, \shO((n-1-k)h), \underline{S^{n-2-k}U^{\vee}(H-h), \shO((n-k)(H-h)-h),} \\
&\underline{ \langle\shO(\ell H)\rangle_{n-k\leq \ell \leq n-1},\shA^{k+3}(H)}, S^{n-k-3}U^{\vee}(H), \langle\shF_{\ell}\rangle_{n-2-k\leq \ell\leq n-3}\Big\rangle\\
&=\Big \langle \shE_{\ell}\rangle_{1\leq \ell\leq n-2-k}, \shO((n-1-k)h), \langle\shO(\ell H)\rangle_{n-k\leq \ell \leq n-1},\shA^{k+3}(H),\\
& \underline{S^{n-2-k}U^{\vee}(H-h), S^{n-k-3}U^{\vee}(H)}, \shO((n-1-k)(H-h)-h),\langle\shF_{\ell}\rangle_{n-2-k\leq \ell\leq n-3}\Big\rangle\\
&=\Big\langle\langle \shE_{\ell}\rangle_{1\leq \ell\leq n-2-k}, \langle\shO(\ell H)\rangle_{n-1-k\leq \ell \leq n-1}, \shA^{k+3}(H), \langle\shF_{\ell}\rangle_{n-3-k\leq \ell\leq n-3}\Big\rangle.  
\end{align*}
This is just the case $k+1$. It is very similar to the argument in the base case above by Lemma \ref{van}(\ref{4ind}) and Lemma \ref{mut}: (i) Exchange $\langle S^{n-2-k}U^{\vee}(H-h), \shO((n-k)(H-h)-h)\rangle$ and $\langle \langle\shO(\ell H)\rangle_{n-k\leq \ell \leq n-1},\shA^{k+3}(H)\rangle$; (ii) Exchange $\shO((n-k)(H-h)-h)$ and $S^{n-k-3}U^{\vee}(H)$; (iii) Right mutate $S^{n-2-k}U^{\vee}(H-h)$ through $S^{n-k-3}U^{\vee}(H)$.
\end{itemize}
\end{proof}
Apply Lemma \ref{4ms} to the final case ($k=n-2$):
\[
\Big\langle \langle \shE_{\ell}\rangle_{1\leq \ell\leq n-2}, \shB_{n-2}, \shA^{2}(H)\Big\rangle= \Big\langle \shE_1, \langle\shO(\ell H)\rangle_{2\leq \ell \leq n-1}, \langle\shF_{\ell}\rangle_{0\leq \ell\leq n-3}\Big\rangle. 
\]
In summary, the outcome of inductive steps 1-4 is
\begin{align}\label{SOD1.2}
\langle \shA(-h), \shA(H-h), \shA, \shA^{1}(H)\rangle=\Big\langle \langle \shO(\ell H)\rangle_{-1\leq \ell\leq n-1},\shH, \langle\shF_{\ell}\rangle_{0\leq \ell\leq n-2}\Big\rangle     \,\, (n\geq 2). 
\end{align}
\end{enumerate}
We will get the SOD (\ref{SOD1mut}) after reorganizing the collections (\ref{SOD1.2}) by Lemma \ref{van} (\ref{ab}).

\end{proof}
\begin{remark}\label{proofn1}
For $N=3$, the inductive steps 2-4 are not needed, and the SOD (\ref{SOD1}) becomes the following:
\begin{align}\label{SOD1s}
 D(X)=\langle \shD, \shO(-h), \shO, \shO(H-h)\rangle. 
 \end{align}
 In this case, we left mutate $\shO(H-h)$ to the far left and left mutate $\shD$ to the far left:
\begin{align}\label{SOD1smut}
 D(X)=\langle \LL_{\shO(-2h)} \shD, \shO(-2h), \shO(-h),\shO\rangle.
 \end{align}
Then Theorem \ref{thm} follows by comparing the SOD (\ref{SOD1smut}) with (\ref{SOD2}) (and actually gives a derived equivalence). 
\end{remark}

\section{Proof of Theorem \ref{thm} by Chess Game Method: $N=2n+1$}
In this section, we use the chess game method developed in \cite{thomas2018notes} and \cite{jiang2021categorical} to prove the theorem \ref{thm}. To achieve this, we need to mutate the SOD (\ref{SOD1mut}) and (\ref{SOD2}) properly. 

\subsection{Mutation of the SOD (\ref{SOD1mut})} \begin{enumerate}
\item Transposition on $\langle \shA(\ell H)\rangle_{n\leq \ell\leq 2n-2}$.
 
For any $0\leq \ell \leq r$, where $r=\left[(n-1)/2\right]$, we can write the SOD of $\shA((n+\ell)H)$ as
\[
\shA((n+\ell)H)=\langle \shA^{2\ell+1}((n+\ell)H), \shA_{n-2\ell-1}((n+\ell)H)\rangle.
\]
By Lemma \ref{van}(\ref{step3}), we can transpose $\langle \shA^{2\ell+1}((n+\ell)H)\rangle_{0\leq \ell\leq r}$ to the far left of $\langle \shA(\ell H)\rangle_{n\leq \ell\leq 2n-2}$ so that we get the following SOD:
\begin{align*}
\langle \shA(\ell H)\rangle_{n\leq \ell\leq 2n-2} =&\Big\langle \langle \shA^{2\ell+1}((n+\ell)H)\rangle_{0\leq \ell\leq r}, \langle \shA_{n-2\ell-1}((n+\ell)H)\rangle_{0\leq \ell\leq r}, \langle \shA(\ell H)\rangle_{n+1+r\leq \ell\leq 2n-2}\rangle\Big\rangle. 
\end{align*}
\item Left mutate $\Big \langle \langle \shA_{n-2\ell-1}((n+\ell)H)\rangle_{0\leq \ell\leq r}, \langle \shA(\ell H)\rangle_{n+1+r\leq \ell\leq 2n-2}\rangle\Big\rangle$
to the far left and then left mutate $\shD$ to the far left:
\begin{align}
D(X)=&\Big \langle \shD_2, \langle \shA_{n-2\ell-1}((-n+1+\ell)H-h)\rangle_{0\leq \ell\leq r}, \langle \shA((\ell-2n+1)H-h)\rangle_{n+1+r\leq \ell\leq 2n-2}, \langle \shO(\ell H)\rangle_{-1\leq \ell\leq n-1},\notag\\
 &\shH, \langle\shF_{\ell}\rangle_{0\leq \ell\leq n-2},S^{n-1}U^{\vee}(H), \langle \shA(kH)\rangle_{2\leq k\leq n-1},\shA^{2\ell+1}((n+\ell)H)\rangle_{0\leq \ell\leq r} \Big\rangle, 
\end{align}
where $\shD_2=\LL_{\Big\langle \langle \shA_{n-2\ell-1}((-n+1+\ell)H-h)\rangle_{0\leq \ell\leq r}, \langle \shA((\ell-2n+1)H-h)\rangle_{n+1+r\leq \ell\leq 2n-2}\Big \rangle}\shD$. 
\end{enumerate}

\subsection{Mutation on the SOD (\ref{SOD2})}

Denote $j_*(p_1^*\shO(ah)\otimes\shO(bH))$ by $\shO(a,b)$ and we use the following SOD for $D(\PP^{2n})$:
\[
D(\PP^{2n})=\langle \shO(kh)\rangle_{-1-n\leq k\leq n-1}\qquad\qquad  (\text{Cf. Figure } \ref{cb1})
\]

Consider the following staircase shape subcategory $\shS_{k}$ $(0\leq k\leq n-2)$ of ${}^{\perp}\pi_1^*D(X_1)$:
\begin{equation*}
\left\langle
\resizebox{\textwidth}{!}{
\begin{tabular}{cccccc}
$\shO(n-1-2k, n+k)$, &$\shO(n-2-2k, n+k)$, &$\shO(n-3-2k, n+k)$, &$\cdots$,   &$\shO(n-2, n+k)$,  &$\shO(n-1, n+k)$\\
                                 &&$\shO(n-3-2k,n+k-1)$, &$\cdots$, &$\shO(n-2, n+k-1)$, &$\shO(n-1, n+k-1)$\\
                               &&$\ddots$ &&$\vdots$ &   $\vdots$\\
                               &&& $\shO(n-3,n+1)$, &$\shO(n-2, n+1)$,  &$\shO(n-1, n+1)$\\
                                                                                                                  &&&&& $\shO(n-1, n)$
\end{tabular}}\right\rangle.
\end{equation*}

Left mutate $\shS_{n-2}$ (blue part in Figure \ref{cb1}) to the far right by applying Lemma \ref{van}(\ref{ab}). 
\begin{align}
&D(X)=\Big\langle \pi_1^*D(X_1), \langle j_*(p_1^*D(\PP^{2n}))(kH)\rangle_{0\leq k \leq n}, \langle\shO(a, n)\rangle_{-1-n\leq a\leq n-2}, \notag\\
&\big \langle \LL_{\shS_{k-1}}\shO(-1-n, n+k),\langle\shO(a, n+k)\rangle_{-n\leq a \leq n-2k-2}\big\rangle_{1\leq k\leq n-2},  \shS_{n-3}\Big\rangle.\label{SOD2mut1}
\end{align}
Note that the red part is $\langle \LL_{\shS_{k-1}}\shO(-1-n, n+k)\rangle_{1\leq k\leq n-2}$ in Figure \ref{cb2}.

Then left mutate $\shS_{n-2}$ to the far left and left mutate $\pi_1^*D(X_1)$ to the far left (Figure \ref{cb2}):
\begin{align}
D(X)=&\Big\langle \shD_1, \shS'_{n-2}, \langle j_*(p_1^*D(\PP^{2n}))(kH)\rangle_{0\leq k \leq n}, \langle\shO(a, n)\rangle_{-1-n\leq a\leq n-2}, \notag\\
&\big \langle \LL_{\shS_{k-1}}\shO(-1-n, n+k),\langle\shO(a, n+k)\rangle_{-n\leq a \leq n-2k-2}\big\rangle_{1\leq k\leq n-2}\Big\rangle,\label{SOD2mutfinal}
\end{align}
where
$$ \shS'=\shS\otimes \shO(-1, 1-2n), \qquad \shD_1=\LL_{\shS'_{n-2}}\pi_1^*D(X_1).$$

\begin{table}[ht!]
\begin{center}
\resizebox{\textwidth}{!}{
\begin{tabular}{C{1cm}|C{1.2cm}|C{1cm}|C{1cm}|C{1cm}|C{1cm}|C{1cm}|C{1cm}|C{1cm}|C{1cm}|C{1cm}|C{1cm}|}
\cline{2-11}
\multicolumn{1}{c|}{$2n-2$} &\cellcolor{red!100}&&&&\cellcolor{blue!100} &\cellcolor{blue!100}&\cellcolor{blue!100}&\cellcolor{blue!100}&\cellcolor{blue!100}&\cellcolor{blue!100}&\cellcolor{blue!100}\\
\cline{2-11}
\multicolumn{1}{c|}{$\vdots$} &\cellcolor{red!100}&&&&&&\cellcolor{blue!100}&\cellcolor{blue!100}&\cellcolor{blue!100}&\cellcolor{blue!100}&\cellcolor{blue!100}\\
\cline{2-11}
$n+1$ &\cellcolor{red!100}&&&&&&&&\cellcolor{blue!100}&\cellcolor{blue!100}&\cellcolor{blue!100}\\
\cline{2-12}
\multicolumn{1}{c|}{$n$} &&&&&&&&&&&\cellcolor{blue!100}\\
\cline{2-12}
\multicolumn{1}{c|}{$\vdots$}&&&&&&&&&&&\\
\cline{2-12}
 
\multicolumn{1}{c|}{1}&&&&&&&&&&&\\
\cline{2-12}
\multicolumn{1}{c|}{0}&&&&&&&&&&&\\
\cline{2-12}
\multicolumn{1}{c}{}&\multicolumn{1}{c}{$-1-n$} & \multicolumn{1}{c}{$-n$} &\multicolumn{1}{c}{$1-n$} & \multicolumn{1}{c}{$2-n$} & \multicolumn{1}{c}{$3-n$} & \multicolumn{1}{c}{$4-n$} & \multicolumn{1}{c}{$\cdots$} &\multicolumn{1}{c}{$n-4$} &\multicolumn{1}{c}{$n-3$}
 &\multicolumn{1}{c}{$n-2$} &\multicolumn{1}{c}{$n-1$}
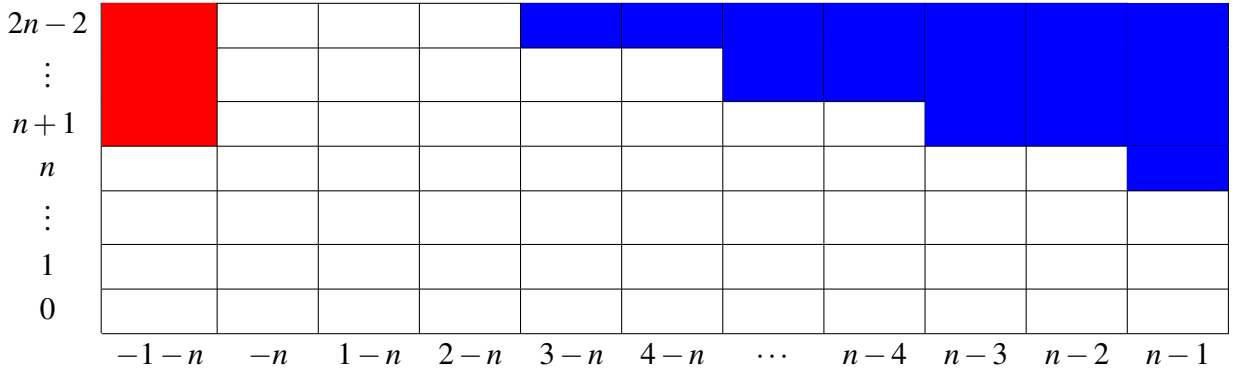
\end{tabular}}\caption{SOD of ${}^{\perp}\pi_1^*D(X_1)$(\ref{SOD2mut1}).  The horizontal (resp. vertical) direction corresponds to $\shO(h)$ (resp. $\shO(H)$).}\label{cb1}
\end{center}
\end{table}

\begin{table}[ht!]
\begin{center}
\resizebox{\textwidth}{!}{
\begin{tabular}{R{1cm}|C{1.2cm}|C{1cm}|C{1cm}|C{1cm}|C{1cm}|C{1cm}|C{1cm}|C{1cm}|C{1cm}|C{1cm}|C{1cm}|}

\cline{2-5}
\multicolumn{1}{c|}{$2n-2$} &\cellcolor{red!100}&&&&\multicolumn{1}{r}{}&\multicolumn{1}{r}{}&\multicolumn{1}{r}{}&\multicolumn{1}{r}{}&\multicolumn{1}{r}{}&\multicolumn{1}{r}{}
&\multicolumn{1}{c}{} \\
\cline{2-7}
\multicolumn{1}{c|}{$\vdots$} &\cellcolor{red!100}&&&&&&\multicolumn{1}{r}{}&\multicolumn{1}{r}{}&\multicolumn{1}{r}{}&\multicolumn{1}{r}{}&\multicolumn{1}{r}{}\\
\cline{2-9}
$n+1$ &\cellcolor{red!100}&&&&&&&&\multicolumn{1}{r}{}&\multicolumn{1}{r}{}&\multicolumn{1}{r}{}\\
\cline{2-11}
\multicolumn{1}{c|}{$n$} &&&&&&&&&&&\multicolumn{1}{r}{}\\
\cline{2-12}

\multicolumn{1}{c|}{$\vdots$} &&&&&&&&&&&\\
\cline{2-12}
\multicolumn{1}{c|}{$1$} &&&&&&&&&&&\\
\cline{2-12}
\multicolumn{1}{c|}{0}&&&&&&&&&&&\\
\cline{2-12}

\multicolumn{1}{r}{$-1$} &\multicolumn{1}{r}{}&\multicolumn{1}{r}{}
&\multicolumn{1}{c|}{} &\cellcolor{blue!100} &\cellcolor{blue!100}&\cellcolor{blue!100}&\cellcolor{blue!100}&\cellcolor{blue!100}&\cellcolor{blue!100}&\cellcolor{blue!100}&\multicolumn{1}{r}{}\\
\cline{5-10}
\multicolumn{1}{c}{$\vdots$} &\multicolumn{1}{r}{}&\multicolumn{1}{r}{}&\multicolumn{1}{r}{}&\multicolumn{1}{r}{}&\multicolumn{1}{r|}{}&\cellcolor{blue!100}&\cellcolor{blue!100}&\cellcolor{blue!100}&\cellcolor{blue!100}&\cellcolor{blue!100}&\multicolumn{1}{r}{}\\
\cline{7-10}
\multicolumn{1}{r}{$-n+2$} &\multicolumn{1}{r}{}&\multicolumn{1}{r}{}&\multicolumn{1}{r}{}
&\multicolumn{1}{r}{}&\multicolumn{1}{r}{}&\multicolumn{1}{r}{}&\multicolumn{1}{r|}{}&\cellcolor{blue!100}&\cellcolor{blue!100}&\cellcolor{blue!100}&\multicolumn{1}{r}{}\\
\cline{9-10}
\multicolumn{1}{r}{$-n+1$} &\multicolumn{1}{r}{}&\multicolumn{1}{r}{}
&\multicolumn{1}{r}{}&\multicolumn{1}{r}{}&\multicolumn{1}{r}{}&\multicolumn{1}{r}{}&\multicolumn{1}{r}{}
&\multicolumn{1}{r}{}&\multicolumn{1}{r|}{}&\cellcolor{blue!100}&\multicolumn{1}{r}{}\\
\multicolumn{1}{c}{}&\multicolumn{1}{c}{$-1-n$} & \multicolumn{1}{c}{$-n$} & 
\multicolumn{1}{c}{$1-n$} & \multicolumn{1}{c}{$2-n$} & \multicolumn{1}{c}{$3-n$} & \multicolumn{1}{c}{$4-n$} & \multicolumn{1}{c}{$\cdots$} &\multicolumn{1}{c}{$n-4$} &\multicolumn{1}{c}{$n-3$}
 &\multicolumn{1}{c}{$n-2$} &\multicolumn{1}{c}{$n-1$}
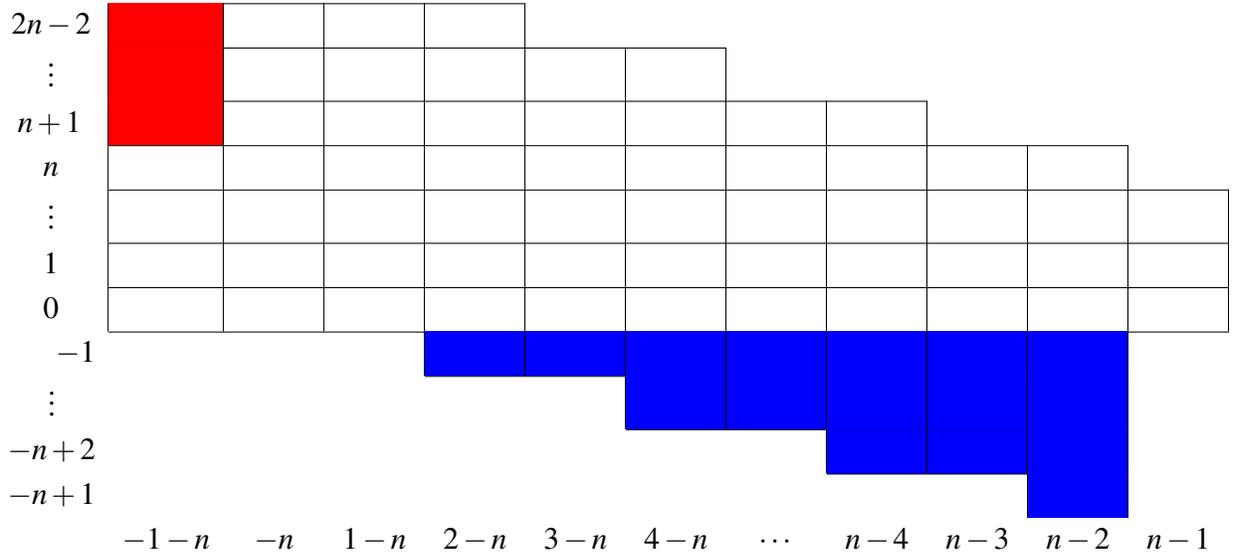
\end{tabular}}\caption{Mutated Chessboard (SOD of ${}^{\perp}\shD_1$ (\ref{SOD2mutfinal}))}\label{cb2}
\end{center}
\end{table}

\subsection{Proof of Theorem \ref{thm}}

To prove that there is a fully-faithful embedding from $D(X_2)$ into $D(X_1)$, it is equivalent to show
\begin{center}
\begin{tikzcd}
\phi: \shD_1 \arrow[r, "i_{\shD_1}"] & D(X) \arrow [r, "\pi_{\shD_2}"]  & \shD_2
\end{tikzcd}
 \end{center}
is fully-faithful, where $i_{\shD_1}$ is the natural embedding and $\pi_{\shD_2}$ is left adjoint to the natural embedding from $\shD_2$ into $D(X)$. That is, for any $x,y \in \shD_1$, 
$$ Hom_{\shD_2}(\phi(x), \phi(y))=Hom_{\shD_1}(x,y).$$ 
By adjunction, 
\begin{align*}
Hom(\phi(x), \phi(y))=Hom(\pi_{\shD_2}i_{\shD_1} x, \pi_{\shD_2}i_{\shD_1}y)=Hom(i_{\shD_1} x, i_{\shD_2} \pi_{\shD_2}i_{\shD_1}y).
\end{align*}
So it is sufficient to show that 
\begin{equation}\label{cone}
 Cone(y\rightarrow \pi_{\shD_2}y)=0. 
\end{equation}
Note that $\pi_{\shD_2}=\LL_{{}^{\perp}\shD_2}$ and to achieve (\ref{cone}), we will show that $Hom(^{{}\perp}\shD_2,y)=0$, which is equivalent to say that $^{\perp}\shD_2$ lies in the whole region in Figure \ref{cb2}. Actually, we will show that $^{\perp}\shD_2$ lies in the subregion consisting of the union of two parts:
\begin{enumerate}
\item[(i)]
$$ \begin{cases}
 x+2y\leq 3n-2;\\
 0\leq y\leq 2n-2;\\
 -n\leq x\leq n-1;
\end{cases}$$

\item[(ii)]$$\begin{cases}
-n\leq x+2y\leq 3n-4;\\
-n\leq x\leq n-2.
\end{cases}$$ 
\end{enumerate}
\begin{proposition}
For any integer $k$,
\begin{align*}
S^kU^{\vee}\in \langle \shO(k-2\ell, \ell)\rangle_{0\leq \ell\leq k}
\end{align*}
\end{proposition}
\begin{proof}
We prove by induction on $k$. Note that the RHS forms an exceptional collection by Lemma \ref{van}(\ref{ab}). 
\begin{enumerate}
\item Base case $k=0$ is trivial. 
\item Assume that we have $S^kU^{\vee}\in \langle \shO(k-2\ell, l)\rangle_{0\leq \ell\leq k}$, and thus 
$$S^kU^{\vee}(H-h)\in \langle \shO(k-2\ell-1, \ell+1)\rangle_{0\leq \ell\leq k}= \langle \shO(k-2\ell, \ell)\rangle_{1\leq \ell\leq k+1}.$$ By short exact sequence (\ref{euler}), $S^{k+1}U^{\vee}\in \langle \shO(k-2\ell, \ell)\rangle_{0\leq \ell\leq k+1}$. 
\end{enumerate}
\end{proof}

\begin{remark}\
We can view that $S^kU^{\vee}$ lies in the \qq{segment} $\{x+2y=k , 0\leq y\leq k, -k\leq x\leq k\}$. Hence, we have the following ($a, b\in\ZZ$) :
\begin{enumerate}
\item[(i)] $S^aU^{\vee}(bH)$ lies in $\{x+2y=a+2b , b\leq y\leq a+b, -a\leq x\leq a\}$;
\item[(ii)] $S^aU^{\vee}(bH-h)$ lies in $\{x+2y=a+2b-1 , b\leq y\leq a+b, -a-1\leq x\leq a-1\}.$
\end{enumerate}
\end{remark}

Now we divide the SOD components of ${}^{\perp}\shD_2$ into two classes: {}

%

\begin{enumerate}
\item $ \Big \langle \langle \shA_{n-2\ell-1}((-n+1+\ell)H-h)\rangle_{0\leq \ell\leq r}, \langle \shA((\ell-2n+1)H-h)\rangle_{n+1+r\leq \ell\leq 2n-2}\Big\rangle$ and

\item $\Big \langle \langle \shO(\ell H)\rangle_{-1\leq \ell\leq n-1}, \shH, 
\langle\shF_{\ell}\rangle_{0\leq \ell\leq n-2},S^{n-1}U^{\vee}(H), \langle \shA(kH)\rangle_{2\leq k\leq n-1}, \langle \shA^{2\ell+1}((n+\ell)H)\rangle_{0\leq \ell\leq r} \Big\rangle$. 

\end{enumerate}

\begin{claim}
The SOD components of Classes (1) (resp. (2)) lie in region (1) (resp. (2)). 
\end{claim}
\begin{proof}
For $S^aU^{\vee}(bH-h)\in \langle \shA_{n-2\ell-1}((-n+1+\ell)H-h)\rangle_{0\leq \ell\leq r}$, it is easy to check that  
$$ -n\leq a+2b-1\leq 3n-4,\quad -n\leq a-1\leq n-2, \quad -a-1\geq -n$$ 
under the conditions:
$$ b=n-2\ell-1, \quad 0\leq a \leq n-1, \quad 0\leq \ell\leq r.$$ 
Hence, $\langle \shA_{n-2\ell-1}((-n+1+\ell)H-h)\rangle_{0\leq \ell\leq r}$ lies in region (1). The other cases can be proved by the same arguments. 
\end{proof}
Therefore (\ref{cone}) holds and the fully faithful functor $\Phi$ is given by the composition of the following functors
\begin{equation}
D(X_1) \xrightarrow{\LL_{\shS'_{n-2}}\pi_1^*}   \shD_1 \xrightarrow{\phi}  \shD_2 \xrightarrow{\RR_{\Big\langle \langle \shA_{n-2\ell-1}((-n+1+\ell)H-h)\rangle_{0\leq \ell\leq r}, \langle \shA((\ell-2n+1)H-h)\rangle_{n+1+r\leq \ell\leq 2n-2}\Big \rangle}}\shD \xrightarrow{\pi_{2*} \RR_{\langle\shA(-h), \shA(H-h)\rangle}} D(X_2).
\end{equation}\label{functor}

\section{Even dimensional case: \texorpdfstring{$N=2n\, (n\geq 2)$}{}} \label{even}
We sketch the process to prove even dimensional cases as follows: 
\begin{enumerate}
\item Step 1: Left mutate $\shA^1((2n-2)H), \shA^1((2n-1)H)$ to the far left, and then left mutate $\pi_2^*D(X_2)$ to the far
left:
\begin{align*}
D(X)=\Big\langle \shD', \shA^1(-h), \shA^1(H-h), \langle \shA(\ell H)\rangle_{0\leq \ell\leq n-1}, \langle \shA^1(kH)\rangle_{n\leq k\leq 2n-3}\Big \rangle,
\end{align*}
where
\[
\shD'=\LL_{\langle \shA^1(-h), \shA^1(H-h)\rangle}\pi_2^*D(X_2).
\]
\item Step 2: Mutations on $\langle \shA^1(-h), \shA^1(H-h), \shA, \shA(H)\rangle$.
\begin{align*}
\langle \shA(-h), \shA(H-h), \shA, \shA(H)\rangle
=\Big \langle \langle\shO(\ell H)\rangle_{-1\leq \ell\leq n-1},\shH', \langle \shF'_{\ell} \rangle_{0\leq k\leq n-3}, S^{n-2}U^{\vee}(H), S^{n-1}U^{\vee}(H)\Big\rangle.
\end{align*}
Here  
\begin{align*}
&\shH'=\begin{cases} 
 \langle  \shO(H-h) \rangle &\mbox{if } n=2;\\
\langle \shO(H-2h), \shO(H-h)\rangle & \mbox{if } n=3;\\
\langle \shO(H-2h), \shO(H-h), \shO(2H-3h)\rangle &\mbox{if } n\geq 4;
 \end{cases}\\
&\shF_{\ell}'=\begin{cases} 
\langle S^{\ell}U^{\vee}(H), \shO((\ell+2)(H-h)),\shO((l+3)(H-h)-h)\rangle &\mbox{if } \ell\leq n-5;\\
\langle S^{\ell}U^{\vee}(H), \shO((\ell+2)(H-h))\rangle &\mbox{if } \ell>n-5.
\end{cases}
\end{align*}

\item Step 3: Mutation on $\langle\shA^1(\ell H)\rangle_{n\leq \ell\leq 2n-3}$.

Let $r'=\left[(n-1)/2\right]-1$, and consider the following SOD for $\shA(aH)$, where $2n-2-r'\leq a \leq 2n-2$,
\begin{align*}
\shA^1((n+\ell)H)=\langle \shA^{2\ell+3}((n+\ell)H), \shA^1_{n-2\ell-3}((n+\ell)H)\rangle. 
\end{align*}

Next left mutate $\Big \langle S^{n-1}U^{\vee}((n-1)H), \langle \shA^1_{n-2\ell-3}((n+\ell)H)\rangle_{0\leq \ell \leq r'}\Big \rangle$ to the far left, and then left mutate $\shD'$ to the far left:
\begin{align*}
&D(X)=\Big \langle \shD'_2, S^{n-1}U^{\vee}(-H-h), \langle\shA^1_{n-2\ell-3}(-(\ell+2)H-h)\rangle_{0\leq \ell\leq r'}, \langle \shA^1((\ell-2n+3)H-h)\rangle_{n+r'+1\leq \ell\leq 2n-3},\\
& \langle\shO(\ell H)\rangle_{-1\leq \ell\leq n-1}, \shH',\langle \shF'_{\ell} \rangle_{0\leq \ell\leq n-3}, S^{n-2}U^{\vee}(H), S^{n-1}U^{\vee}(H), \langle\shA(\ell H)\rangle_{2\leq \ell\leq n-2},\langle\shA^{2\ell-3}((n+\ell)H)\rangle_{-1\leq \ell\leq r'}\Big\rangle, 
\end{align*}
where
\[
\shD_2'=\LL_{\Big \langle S^{n-1}U^{\vee}((n-1)H), \langle \shA^1_{n-2\ell-3}((n+\ell)H)\rangle_{0\leq \ell \leq r'}\Big \rangle}\shD'.
\]

\item Step 4: Mutation on the SOD (\ref{SOD2}) of $D(X)$.

Mutate blue part to the far left, and mutate $D(X_1)$ to the far left (Figure \ref{cb3}). 

\item Step 5: Conclude the main theorem via analyzing ${}^{\perp}\shD'_2$.  
\end{enumerate}

\begin{table}[h!]
\begin{center}
\resizebox{\textwidth}{!}{
\begin{tabular}{R{1cm}|C{1cm}|C{1cm}|C{1cm}|C{1cm}|C{1cm}|C{1cm}|C{1cm}|C{1cm}|C{1cm}|C{1cm}|}

\cline{2-4}
\multicolumn{1}{c|}{$2n-3$} &\cellcolor{red!100}&&&\multicolumn{1}{r}{}&\multicolumn{1}{r}{}&\multicolumn{1}{r}{}&\multicolumn{1}{r}{}&\multicolumn{1}{r}{}&\multicolumn{1}{r}{}
&\multicolumn{1}{c}{} \\
\cline{2-6}
\multicolumn{1}{c|}{$\vdots$} &\cellcolor{red!100}&&&&&\multicolumn{1}{r}{}&\multicolumn{1}{r}{}&\multicolumn{1}{r}{}&\multicolumn{1}{r}{}&\multicolumn{1}{r}{}\\
\cline{2-8}
$n+1$ &\cellcolor{red!100}&&&&&&&\multicolumn{1}{r}{}&\multicolumn{1}{r}{}&\multicolumn{1}{r}{}\\
\cline{2-10}
\multicolumn{1}{c|}{$n$} &&&&&&&&&&\multicolumn{1}{r}{}\\
\cline{2-11}

\multicolumn{1}{c|}{$\vdots$} &&&&&&&&&&\\
\cline{2-11}
\multicolumn{1}{c|}{$1$} &&&&&&&&&&\\
\cline{2-11}
\multicolumn{1}{c|}{0}&&&&&&&&&&\\
\cline{2-11}

\multicolumn{1}{r}{$-1$} &\multicolumn{1}{r}{}
&\multicolumn{1}{c|}{} &\cellcolor{blue!100} &\cellcolor{blue!100}&\cellcolor{blue!100}&\cellcolor{blue!100}&\cellcolor{blue!100}&\cellcolor{blue!100}&\cellcolor{blue!100}&\multicolumn{1}{r}{}\\
\cline{5-10}
\multicolumn{1}{c}{$\vdots$} &\multicolumn{1}{r}{}&\multicolumn{1}{r}{}&\multicolumn{1}{r}{}&\multicolumn{1}{r|}{}&\cellcolor{blue!100}&\cellcolor{blue!100}&\cellcolor{blue!100}&\cellcolor{blue!100}&\cellcolor{blue!100}&\multicolumn{1}{r}{}\\
\cline{7-10}
\multicolumn{1}{r}{$-n+2$} &\multicolumn{1}{r}{}&\multicolumn{1}{r}{}
&\multicolumn{1}{r}{}&\multicolumn{1}{r}{}&\multicolumn{1}{r}{}&\multicolumn{1}{r|}{}&\cellcolor{blue!100}&\cellcolor{blue!100}&\cellcolor{blue!100}&\multicolumn{1}{r}{}\\
\cline{9-10}
\multicolumn{1}{r}{$-n+1$} &\multicolumn{1}{r}{}
&\multicolumn{1}{r}{}&\multicolumn{1}{r}{}&\multicolumn{1}{r}{}&\multicolumn{1}{r}{}&\multicolumn{1}{r}{}
&\multicolumn{1}{r}{}&\multicolumn{1}{r|}{}&\cellcolor{blue!100}&\multicolumn{1}{r}{}\\
\multicolumn{1}{c}{}& \multicolumn{1}{c}{$-n$} & 
\multicolumn{1}{c}{$1-n$} & \multicolumn{1}{c}{$2-n$} & \multicolumn{1}{c}{$3-n$} & \multicolumn{1}{c}{$4-n$} & \multicolumn{1}{c}{$\cdots$} &\multicolumn{1}{c}{$n-4$} &\multicolumn{1}{c}{$n-3$}
 &\multicolumn{1}{c}{$n-2$} &\multicolumn{1}{c}{$n-1$}
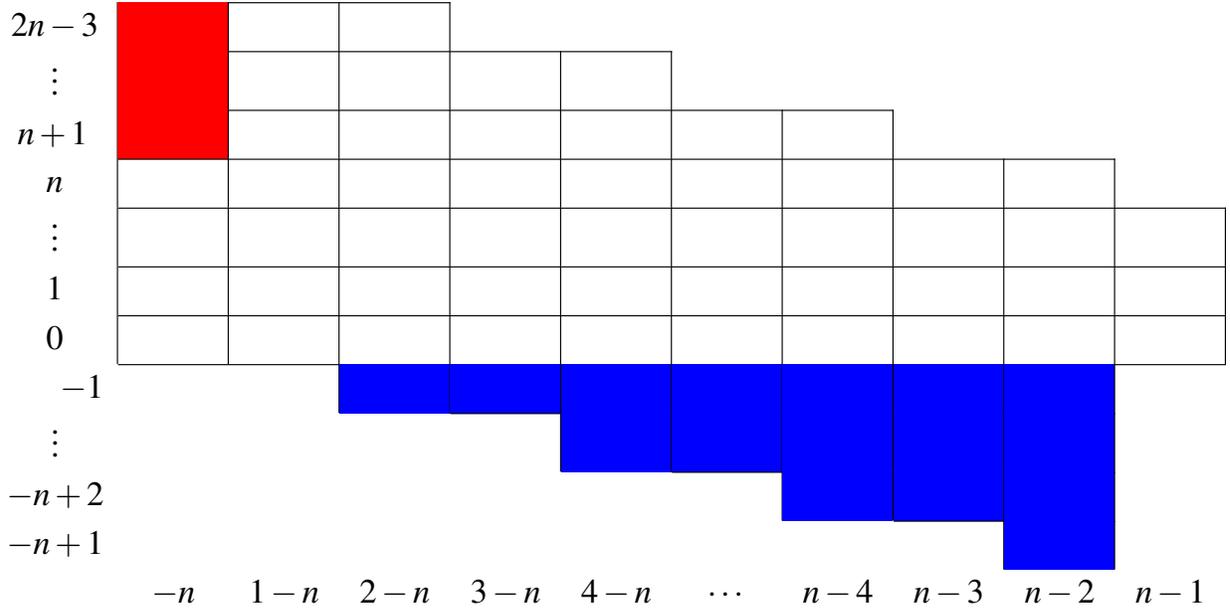
\end{tabular}}\caption{Mutated Chessboard (${}^{\perp}\shD_1'$)}\label{cb3}
\end{center}
\end{table}

\begin{remark}
It should be noted that when $n=2, N=4$,  there is no $U^{\vee}(H)$ in the final SOD of $D(X)$ since we have mutated it to the far left at the beginning of step 3:  
$$ D(X)=\langle \LL_{U^{\vee}(-H-h)}\shD', U^{\vee}(-H-h), \shO(-h), \shO, \shO(h), \shO(H-h), \shO(H)\rangle. $$
\end{remark}

\clearpage

\appendix

\renewcommand\thesubsection{A\arabic{subsection}}
\renewcommand\thetheorem{A\arabic{theorem}}

\section{Background on mutations and Borel-Weil-Bott}

\subsection{Semiorthogonal decompositions and mutations} \label{appa}
A \emph{semiorthogonal decomposition} (SOD) of a triangulated category $\shT$, written as:
	\begin{equation} \label{SOD:T}
	\shT = \langle \shT_1, \shT_2, \ldots, \shT_{n} \rangle,
	\end{equation}
	 is formed by a sequence of full triangulated subcategories $\shT_1, \ldots, \shT_n$ of $\shT$ such that
	\begin{enumerate}
	        \item the natural inclusion functor $\iota_{\shT_i}: \shT_i \hookrightarrow \shT$ admits both right and left adjoints. 
		\item $\Hom_\shT (t_k ,t_\ell) = 0$ for all $t_k \in \shT_k$ and $t_\ell \in \shT_\ell$, if $k > \ell$, and

		\item $\shT_1, \shT_2, \ldots, \shT_{n}$ generates $\shT$, \emph{i.e., } the smallest triangulated category containing $\shT_1, \shT_2, \ldots, \shT_{n}$ that is closed under shifting and taking cones. 
\end{enumerate}

The subcategory $\shT_i$ satisfying the condition (1) is called \emph{admissible}.  A sequence $\shT_1, \ldots, \shT_n$ satisfying the conditions (1)\,\&\,(2) is called \emph{a semiorthogonal collection}. 
An object $E$ in $\shT$ is called \emph{exceptional} if $\Hom(E, E)\cong k$ and $\Hom(E, E[i])=0$ for each non-zero integer $i$. The SOD (\ref{SOD:T}) is called \emph{a full exceptional collection} of $\shT$ if each $\shT_i$ is generated by an exceptional object.   

Suppose $\shT'$ is an admissible subcategory of a triangulated category $\shT$. Then denote 
	$$\shT'^\perp := \{ T \in \shT \mid \Hom(\shT',T) = 0\}, \qquad {}^\perp \shT' :=\{ T \in \shT \mid \Hom(T, \shA) = 0\}$$ 
to be the {\em right} and respectively {\em left orthogonal} of $\shT'$ inside $\shT$. $\shT'^\perp$ and ${}^\perp \shT'$ are both admissible, and we have the SOD $\shT = \langle \shT'^\perp, \shT' \rangle =  \langle \shT', {}^\perp \shT'\rangle$.

Starting with a SOD, one can obtain a whole collection of new decompositions by \emph{mutations}.  Let $\shT'$ be an admissible subcategory of a triangulated category $\shT$. Then the functor $\LL_{\shT'}: = i_{\shT'^\perp} i^*_{\shT'^{\perp}} : \shT \to \shT$  (resp. $\RR_{\shT'} : =  i_{{}^\perp \shT'} i^!_{ {}^\perp \shT'}: \shT \to \shT$) is called the \emph{left} (resp. {\em right}) {\em mutation through $\shT'$}, where $ i^*_{\shT'^{\perp}}$ (resp.  $i^!_{ {}^\perp \shT'}$) is the left (resp. right) adjoint functor to the inclusion $i_{\shT'^\perp}:  \shT'^\perp \hookrightarrow \shT$. The following results are standard, see  \cite{kuznetsov2010derived}, \cite{bondal1989representation} and \cite{kuznetsov2007homological}. 
\begin{lemma} \label{lem:mut} Let $\shT'$ and $\shT_1, \ldots, \shT_n$ be admissible subcategories of a triangulated category $\shT$ where $n \ge 2$ is an integer. 
\begin{enumerate}
	\item For any $b \in \shT$, there are distinguished triangles
		$$ i_{\shT'} i^!_{\shT'} (b) \to b \to \LL_{\shT'} \,b \xrightarrow{[1]}{},\qquad  \RR_{\shT'} \,b \to b \to  i_{\shT'} i^*_{\shT'} (b)  \xrightarrow{[1]}{}.$$
		
	\item $(\LL_{\shT'})\,|_{\shT'} = 0$ and $(\RR_{\shT'})\,|_{\shT'} = 0$ are the zero functors, and $(\LL_{\shT'})\,|_{\shT'^\perp} = \Id_{\shT'^\perp}: \shT'^\perp \to \shT'^\perp$, $(\RR_{\shT'})\,|_{{}^\perp \shT'} = \Id_{{}^\perp \shT'}: {}^\perp \shT' \to {}^\perp \shT'$ are identity functors. Furthermore $(\LL_{\shT'})\,|_{{}^\perp \shT'} : {}^\perp \shT' \to \shT'^\perp$ and $(\RR_{\shT'})\,|_{\shT'^\perp } : \shT'^\perp \to {}^\perp \shT' $ are mutually inverse equivalences of categories.
	\item If $\shT$ admits a Serre functor $S$, then  $\LL_{\shT'}\,|_{{}^\perp\shT'}=S_{\shT}\,|_{{}^\perp\shT'}$ and $\RR_{\shT'}\,|_{\shT'^{\perp}}=S^{-1}_{\shT}\,|_{\shT'^{\perp}}$. 

	\item \label{lem:mut:span} If $\shT=\langle\shT_1, \ldots, \shT_{k-1}, \shT_k, \shT_{k+1}, \ldots, \shT_n\rangle$, then 
		 \begin{align*}
		\langle \shT_1, \ldots, \shT_{k-1}, \shT_k, \shT_{k+1}, \ldots, \shT_n\rangle& = \langle \shT_1, \ldots, \shT_{k-2}, \LL_{\shT_{k-1}} (\shT_k), \shT_{k-1}, \shT_{k+1}, \ldots, \shT_n \rangle\\
		&=\langle \shT_1, \ldots, \shT_{k-1}, \shT_{k}, \RR_{\shT_{k}} (\shT_{k-1}), \shT_{k-1}, \shT_{k+1}, \ldots, \shT_n\rangle.
		\end{align*}

\end{enumerate}
\end{lemma}\label{defmut}
In particular, if $\shT'$ is generated by only one object $E$, then for $b\in \shT$, 
\begin{align} 
& \LL_E(b)=Cone(\RR Hom(E, b)\otimes E\xrightarrow{ev} b)\label{lmut}\\
&\RR_E(b)=Cone (b\xrightarrow{ev^{\vee}} \RR Hom(b, E)^{\vee}\otimes E^{\vee})[-1].
\end{align}

\subsection{Borel-Weil-Bott Theorem}
We will use the following special case of the Borel-Weil-Bott (BWB) theorem repeatedly. Recall that for any non-increasing sequence of integers $(a_1, a_2)$, one can get an associated $Schur$ $functor$ $\Sigma^{a_1, a_2}$. The readers can refer to section 2 of \cite{kuznetsov2008exceptional} for the general statement of BWB and relevant background. 
\begin{theorem}[Special case of BWB]\label{BWB}
For any integers $a_1\geq a_2$, 
\[
H^{\bullet}(Gr(2,N), \Sigma^{a_1, a_2} U^{\vee})=0 
\]
if $1-N\leq a_1\leq -2$ or $2-N\leq a_2\leq -1$. 
\end{theorem} 
Also, we will use the following projection formula frequently. 
\begin{lemma}\label{pfp}
For any integer $a$,
\begin{align*}
\RR p_{2*}\shO(ah)=\begin{cases}  S^a U^{\vee}  \quad &\text{if} \,\,\,a\geq 0,\\
0 \quad &\text{if} \,\,\, a=-1,\\
  S^{-a-2}U\otimes \shO(-H)[-1] \quad &\text{if} \,\,\,  a\leq -2.\\
 \end{cases}
\end{align*}
\end{lemma}

\begin{proof}
The first and second lines are easy. For the third line, 
\begin{align*}
&\mathcal{H}om (\RR p_{2*}\shO(ah),\shO)=\RR p_{2*} \mathcal{H}om(\shO(ah), p_2^!\shO)=\RR p_{2*} \mathcal{H}om(\shO(ah), \shO(H-2h)[1])\\
&=\RR p_{2*}\mathcal{H}om(\shO, \shO(H+(-a-2)h)[1])=\shO(H)\otimes S^{-a-2}U^*[1].
\end{align*}
\end{proof}

\subsection{Proof of Lemma \ref{van}}\label{proofofvan}
\begin{proof}[Proof of Lemma \ref{van} (\ref{2ind}):]
Follow the proof of (1), and it is sufficient to show for $k+2\leq a\leq n-1$ and $0\leq b\leq k$, 
\begin{align*}
&\RR Hom_E(S^aU^{\vee}, \shO((k+1)h))=\Ext^{\bullet}_{Gr(2, N)}( S^aU^{\vee}, S^{k+1}U^{\vee})=0,\\
& \RR Hom_E(S^aU^{\vee}(H), \shO(kh))=\Ext^{\bullet}_{Gr(2,N)}( S^aU^{\vee}(H), S^{k}U^{\vee})=0,\\
& \RR Hom_E(S^bU^{\vee}(H), \shO((k+1)h))=\Ext^{\bullet}_{Gr(2,N)}(S^bU^{\vee}(H), S^{k+1}U^{\vee})=0,\\
&\RR Hom_E(S^bU^{\vee}(2H), \shO(kh))=\Ext^{\bullet}_{Gr(2,N)}(S^bU^{\vee}(2H), S^{k}U^{\vee})=0.
\end{align*}
These all hold by the SOD of $D(Gr(2, N))$. 
\end{proof}

\begin{proof}[Proof of Lemma \ref{van} (\ref{3ind}):]
It is sufficient to show that
\begin{align*}
& \Ext^{\bullet} _E(\shO(\ell H), S^{k-1}U^{\vee}(H-h))=\Ext^{\bullet-1}_{Gr(2, N)}(S^{k-1}U(-H), S^{l-1}U(-H))=0;\\
& \Ext^{\bullet} _E(\shO(\ell H), S^{k-1}U^{\vee}(-2h))=\Ext^{\bullet-1}_{Gr(2, N)}(S^{k-1}U, S^lU)=0;\\
& \Ext^{\bullet} _E(S^l U^{\vee}(H-2h), S^{k-1}U^{\vee}(H-h))=\Ext^{\bullet}_{Gr(2, N)}(S^l U^{\vee}\otimes S^{k-1}U, p_{2*}\shO(-h))=0;\\
& \Ext^{\bullet} _E(S^l U^{\vee}(2H-h), S^{k-1}U^{\vee}(H-h))=\Ext^{\bullet}_{Gr(2, N)}(S^l U^{\vee}(H), S^{k-1}U^{\vee})=0.
\end{align*}
The third line uses the projection formula (Lemma \ref{pfp}), and the rest uses the SOD of $D(Gr(2, N))$. 
\end{proof}

\begin{proof}[Proof of Lemma \ref{van} (\ref{4ind}):]
It is sufficient to show for $0\leq a\leq n-k-2$ and $0\leq b\leq n-k-3$, 
\begin{align*}
& \RR Hom_E(S^{n-2-k} U^{\vee}(H-h), \shO(\ell H))=\Ext^{\bullet}_{Gr(2, N)}(S^{n-2-k} U^{\vee}(H), S^{l+1}U^{\vee})=0;\\
& \RR Hom_E(S^{n-2-k} U^{\vee}(2H), \shO(\ell H))=\Ext^{\bullet}_{Gr(2, N)}(S^{n-2-k} U^{\vee}(2H), S^{l}U^{\vee})=0;\\
&\RR Hom_E(S^{n-k} U^{\vee}(H-h), S^aU^{\vee}(H))=\Ext^{\bullet}_{Gr(2, N)}(S^{n-k} U^{\vee}(H), S^a U^{\vee}\otimes U^{\vee})=0;\\
&\RR Hom_E(S^{n-k} U^{\vee}(2H), S^aU^{\vee}(H))=\Ext^{\bullet}_{Gr(2, N)}(S^{n-k} U^{\vee}(H), S^a U^{\vee})=0;\\
&\RR Hom_E(\shO((n-k)(H-h)-h, S^bU^{\vee}(H))=\Ext^{\bullet}_{Gr(2, N)}(\shO((n-1-k)H), S^bU^{\vee}\otimes S^{n-k+1}U^{\vee})=0;\\
&\RR Hom_E(\shO((n-k)(H-h), S^bU^{\vee})=\Ext^{\bullet}_{Gr(2, N)}(\shO((n-k)H), S^bU^{\vee}\otimes S^{n-k}U^{\vee})=0.
\end{align*}
All these can be shown by theorem \ref{BWB} and the SOD of $D(Gr(2, N))$. 
\end{proof}

\begin{proof}[Proof of Lemma \ref{van} (\ref{step3}):]
It is sufficient to show for $n-2\ell-1\leq a\leq n-1$ and $0\leq b\leq n-2k-2$,  
\begin{align*}
& \Ext^{\bullet}_E(S^aU^{\vee}((\ell-k)H), S^bU^{\vee})=H^{\bullet}(Gr(2,N), S^aU\otimes S^bU^{\vee}((k-l)H))=0;\\
& \Ext^{\bullet}_E(S^aU^{\vee}((\ell-k+1)H), S^bU^{\vee}(-h))=0.
\end{align*}
The second line uses the projection formula (Lemma \ref{pfp}).  For the first line, by Littlewood-Richardson rule, $$S^aU\otimes S^bU^{\vee}((k-\ell)H=\bigoplus_{t=0}^b \Sigma^{b-t+k-\ell, -a+t+k-\ell}U^{\vee}.$$
It is easy to check that $-a+t+k-\ell\in [2-N, -1]$ whenever $0\leq t\leq b, n-2\ell-1\leq a\leq n-1, 0\leq b\leq n-2k-2, 0\leq \ell<k \leq r$. 
\end{proof}

\begin{proof}[Proof of Lemma \ref{van} (\ref{ab})]
It is sufficient to show that
\begin{align*}
& \Ext^{\bullet}_E(\shO(ah), \shO(bH))=H^{\bullet-1}(Gr(2,N), \Sigma^{b-1, b-a+1}U^{\vee})=0;\\
& \Ext^{\bullet}_E(\shO((a+1)h), \shO((b-1)H)=H^{\bullet-1}(Gr(2,N),\Sigma^{b-2, b-a-1}U^{\vee})=0.
\end{align*}

These hold whenever condition (i) or (ii) by theorem \ref{BWB}. 
\end{proof}

\subsection{Proof of Lemma \ref{mut}}\label{proofofmut}
\begin{proof}[Proof of Lemma \ref{mut}:] We will give the proof of (1) only. The others are very similar. We need to evaluate $$\Ext^{\bullet}(S^{k-1}U^{\vee}(H-h), S^{k}U^{\vee}).$$  For that, we need to compute
(i) $\Ext_E^{\bullet}(S^{k-1}U^{\vee}(2H), S^{k}U^{\vee})$ and (ii) $\Ext_E^{\bullet}(S^{k-1}U^{\vee}(H-h), S^{k}U^{\vee})$ by the distinguished triangle (\ref{att}).

Actually, by the SOD of $D(Gr(2,N))$
\begin{enumerate}\item[(i)] $\Ext_E^{\bullet}(S^{k-1}U^{\vee}(2H), S^{k}U^{\vee})=\Ext_{Gr(2,N)}^{\bullet}(S^{k-1}U^{\vee}(2H), S^{k}U^{\vee})=0.$ 
\item[(ii)]  $\Ext_E^{\bullet}(S^{k-1}U^{\vee}(H-h), S^{k}U^{\vee})=\Ext_{Gr(2, N)}^{\bullet}(S^{k-1}U^{\vee}(H), S^kU^{\vee}\otimes U^{\vee})=\\
\Ext_{Gr(2, N)}^{\bullet}(S^{k-1}U^{\vee}(H), S^{k+1}U^{\vee}\oplus S^{k-1}U^{\vee}(H))=\CC[0]$.
\end{enumerate}

By  (\ref{lmut}),
\begin{align*}
& \LL_{S^{k-1}U^{\vee}((k-1)(H-h))}S^kU^{\vee}\\
&=Cone(\RR Hom(S^{k-1}U^{\vee}((k-1)(H-h)), S^kU^{\vee})\otimes S^{k-1}U^{\vee}((k-1)(H-h))\longrightarrow S^kU^{\vee})\\
& =Cone(S^{k-1}U^{\vee}((k-1)(H-h))\longrightarrow S^kU^{\vee}) =\shO(kh).
 \end{align*}
 
 The last equality uses the short exact sequence (\ref{euler}).

\end{proof}

\end{document}